\documentclass[11pt]{article}

\usepackage{amsmath,amsthm,comment,url}
\usepackage{amssymb}
\usepackage[left=1in,right=1in,top=1in,bottom=1in]{geometry}

\newcommand{\R}{\mathbb{R}}

\newcommand{\E}{\mathbb{E}}

\newcommand{\pts}{\mathcal P}
\newcommand{\lines}{\mathcal L}
\newcommand{\bi}[1]{\mathcal{B}(#1)}

\newtheorem{theorem}{Theorem}[section]
\newtheorem{lemma}[theorem]{Lemma}

\newtheorem{proposition}[theorem]{Proposition}

\usepackage{graphicx,psfrag}
\usepackage[rflt]{floatflt}

\newcommand{\placefig}[2]
        {\includegraphics[width=#2]{#1.eps}}

\theoremstyle{definition}

\def\eps{{\varepsilon}}

\newcommand{\parag}[1]{\vspace{2mm}

\noindent{\bf #1} }

\newcommand{\prr}{{\mathrm{P\hspace{-0.05em}r}}}
\newcommand{\EE}{\mathcal{E}}

\begin{document}

\title{Local Properties via Color Energy Graphs and Forbidden Configurations\footnote{This research project was done as part of the 2018 CUNY Combinatorics REU, supported by NSF grant DMS-1710305.}}

\author{
Sara Fish\thanks{California Institute of Technology, Pasadena, CA, USA
{\sl sfish@caltech.edu}. Supported by Caltech's Summer Undergraduate Research
Fellowships (SURF) program.}
\and
Cosmin Pohoata\thanks{California Institute of Technology, Pasadena, CA, USA
{\sl apohoata@caltech.edu}.}
\and
Adam Sheffer\thanks{Department of Mathematics, Baruch College, City University of New York, NY, USA.
{\sl adamsh@gmail.com}. Supported by NSF award DMS-1710305 and PSC-CUNY award 61666-00-49.}}

\date{}
\maketitle
\begin{abstract}
The local properties problem of Erd\H os and Shelah generalizes many Ramsey problems and some distinct distances problems.
In this work, we derive a variety of new bounds for the local properties problem and its variants.
We do this by continuing to develop the color energy technique --- a variant of the concept of additive energy from Additive Combinatorics.
In particular, we generalize the concept of color energy to higher color energies, and combine these with Extremal Graph Theory results about graphs with no cycles or subdivisions of size $k$.
\end{abstract}

\section{Introduction}

Erd\H os and Shelah \cite[Section V]{Erdos74} suggested a problem concerning local properties of a graph.
Consider a complete graph $K_n=(V,E)$, a set of colors $C$, and a coloring of the edges $\chi:C \to E$.
For parameters $k$ and $\ell$, assume that every induced subgraph over $k$ vertices of $V$ contains at least $\ell$ colors.
We define $f(n,k,\ell)$ as the minimum size $C$ can have while still satisfying the above property.
For example, $f(n,3,3)$ is the minimum number of colors in an edge coloring of $K_n$ where every triangle contains three distinct colors.
In this case no vertex can be adjacent to two edges of the same color, so $f(n,3,3)\ge n-1$.
The parameters $k$ and $\ell$ are usually considered to be constants not depending on $n$.

One reason for studying the local properties problem is that it generalizes many Ramsey problems, and also some distinct distances problems.
Moreover, this problem has a rich history and attracted the interest of various well-known mathematicians.
For a survey of some of this history, see for example a previous work by the second and third authors \cite{PS18}.

The \emph{linear threshold} of the local properties problem is the smallest $\ell$ (for a given $k$) for which $f(n,k,\ell)=\Omega(n)$.
Similarly, the \emph{quadratic threshold} is the smallest $\ell$ for which $f(n,k,\ell)=\Omega(n^2)$.
Erd\H os and Gy\'arf\'as \cite{EG97} proved that the quadratic threshold is $\ell = \binom{k}{2}-\lfloor \frac{k}{2}\rfloor +2$ and that the linear threshold is $\ell = \binom{k}{2}-k +3$.
Conlon, Fox, Lee, and Sudakov \cite{CFLS14} showed that the \emph{polynomial threshold} is $\ell=k$ (the smallest $\ell$ for which $f(n,k,\ell)=\Omega(n^\eps)$ for some $\eps>0$).
Recently, a family of additional thresholds appeared in \cite{PS18}: For each integer $m\ge 2$ there exists $0\le c_m\le m+1$ such that the threshold for having $f(n,k,\ell)=\Omega\left(n^{(m+1)/m}\right)$ is $\ell = \binom{k}{2}-m\cdot \lfloor \frac{k}{m+1}\rfloor +c_m$.
Deriving this family of polynomial thresholds was based on introducing a new tool called the \emph{color energy} of a graph.
This tool is a variant of the concept of additive energy from Additive Combinatorics (for example, see \cite{TV06}).

In the current work we derive a variety of new bounds for the local properties problem and its variants.
First, we introduce another family of thresholds for $f(n,k,\ell)$.

\begin{theorem} \label{th:threshold}
(a) For any integers $2 \leq m \leq k/2$,
\[ f\left(n,k,\binom{k}{2} - m(k - m) + 2\right) = \Omega\left(n^{1/m}\right). \]

(b) For any integer $t\ge3$ and $k= \binom{t+1}{2}$,
\[ f\left(n,k,\binom{k}{2} - t(t-1) +1 \right) = \Omega\left(n^{1/2+1/(4t-6)}\right). \]
\end{theorem}

In the first part of the theorem, we are studying the thresholds for $f(n,k,\ell)=\Omega(n^{1/m})$ for $m\ge 2$.
The following upper bound was obtained by Erd\H os and Gy\'arf\'as \cite{EG97}.
\begin{equation} \label{eq:ProbLower}
f\left(n,k,\ell\right) = O\left(n^{\frac{k-2}{\binom{k}{2}-\ell+1}}\right).
\end{equation}
This bound implies that the threshold for $\Omega(n^{1/m})$ is at least $\ell = \binom{k}{2} - m(k - 2) + 2$.
Thus, Theorem \ref{th:threshold}(a) establishes a tight threshold for $\Omega(n^{1/2})$.
When $m\ge 3$, there remains a gap
\[ \binom{k}{2} - m(k - 2) + 2 \le \ell \le \binom{k}{2} - m(k - m) + 2. \]

We extend the aforementioned technique of color energy, using it to derive several new bounds.
As a first example, recall that the quadratic threshold is $\ell = \binom{k}{2}-\lfloor \frac{k}{2}\rfloor +2$.
Erd\H os and Gy\'arf\'as \cite{EG97} asked what happens when we move one away from the quadratic threshold.
That is, they studied the case of $\ell = \binom{k}{2}-\lfloor \frac{k}{2}\rfloor +1$, and derived the bound $f\left(n,k,\ell\right) =\Omega(n^{4/3})$.
This was later improved in \cite{PS18} to $f\left(n,k,\ell\right) =\Omega(n^{3/2})$.
Using our extended color energy technique, we show that $f\left(n,k,\ell\right)$ becomes arbitrarily close to $n^2$ as $k$ grows.

\begin{theorem} \label{th:OneFromQuad}
For every $k \geq 8$ that is divisible by four, we have
\[ f\left(n,k,\binom{k}{2} - k/2 + 1\right) = \Omega\left(n^{2 - 8/k}\right). \]
\end{theorem}
From \eqref{eq:ProbLower}, we obtain
\[ f\left(n,k,\binom{k}{2} - k/2 + 1\right) = O\left(n^{2- 4/k}\right). \]

The revised energy technique can be used to obtain additional new bounds.
As another example, we derive the following result.

\begin{theorem} \label{th:ThreeCycles}
$\displaystyle  f\left(n,24,\binom{24}{2} - 15\right) = \Omega\left(n^{9/8}\right)$.
\end{theorem}

To the best of our knowledge, none of the previous techniques lead to a non-trivial lower bound for $\displaystyle  f\left(n,24,\binom{24}{2} - 15\right)$.
From \eqref{eq:ProbLower} we obtain $f\left(n,24,\binom{24}{2} - 15\right) = O\left(n^{11/7}\right)$.
It seems that similar new bounds could be obtained using the same methods as in the proof of Theorem \ref{th:ThreeCycles}.
However, so far we were not able to derive a family of such bounds --- each bound requires a separate technical proof.
We thus leave the further exploration of this technique for future works.

We also use our techniques to derive new bounds for the ``arithmetic'' variant of the local properties problem.
In this variant we have a set $A$ of $n$ real numbers.
We define the \emph{difference set} of $A$ as
\[ A-A = \{ a-a' :\ a,a'\in A \text{ and } a-a'> 0\}. \]
The standard definition of a difference set includes also non-positive differences.
This does not change the asymptotic size of $A-A$, and thus does not affect our problem.
On the other hand, ignoring non-positive differences makes the problem more natural and easier to study.

Let $g(n,k,\ell)$ denote the minimum size of $A-A$ in any set $A$ of $n$ real numbers that satisfies the following property:
Every subset $A'\subset A$ of size $k$ satisfies $|A'-A'|\ge \ell$.
Equivalently, this is the original local properties problem $f(n,k,\ell)$ when every vertex corresponds to an element of $A$ and the color of an edge $(a,a')$ is $|a-a'|$.

A discussion about the distinction between $f(n,k,\ell)$ and $g(n,k,\ell)$ can be found in \cite{PS18}.
Note that every lower bound for $f(n,k,\ell)$ is also a lower bound for $g(n,k,\ell)$.
Using our tools we derive significantly stronger lower bounds for $g(n,k,\ell)$.
For example, while the linear threshold of $f(n,k,\ell)$ is $\ell = \binom{k}{2}-k +3$, we get that $g(n,k,\ell)$ is super-linear also when $\ell \approx \binom{k}{2}/2$.

\begin{theorem} \label{th:ArithLower}
For all $k > r \geq 2$,
\begin{align*}
g\left(n,2rk,\binom{2rk}{2} - \binom{2k}{2}\cdot\left[ \binom{r}{2} + (r-1) \right] + 1\right) &= \Omega\left(n^{\frac{r}{r-1}\cdot \frac{k-1}{k}}\right).
\end{align*}
\end{theorem}

For example, by setting $r=2$ in Theorem \ref{th:ArithLower}, we get that for every even $k\ge 4$,
\[ g\left(n,k,\binom{k}{2} - 2\cdot \binom{k/2}{2} + 1\right) = \Omega\left(n^{2-\frac{8}{k}}\right). \]
For large $k$, the expression $2\cdot \binom{k/2}{2}$ is almost half of $\binom{k}{2}$.
That is, the number of allowed difference repetitions is about half of the total number of pairs.
This behavior is very different than the behavior of $f(n,k,\ell)$, where the linear threshold occurs already when there are about $k$ repetitions.
As we increase $r$ in Theorem \ref{th:ArithLower}, the number of allowed repetitions increases while the lower bound for the number of differences decreases.

Finally, we observe a simple upper bound for $g(n,k,\ell)$.
\begin{proposition} \label{prop:ArithLower}
For every $\eps > 0$, any sufficiently large $c$ satisfies the following.
For every sufficiently large integer $k$,
\[ g\left(n,k,c \cdot k \cdot \log^{1/4 - \eps} k\right) = n\cdot 2^{O(\sqrt{\log n})}. \]
\end{proposition}

\parag{Our approach.}
Consider a graph $G = (V, E)$, a set of colors $C$, and a function $\chi: E \to C$.
For an edge $e=(v_1,v_2)\in E$, we also write $\chi(v_1,v_2) = \chi(v_2,v_1)= \chi(e)$.
We define the \emph{color energy} of $G$ as
\begin{equation} \label{eq:EnergyDeg}
\E(G) = \left|\left\{ (v_1,v_2,v_3,v_4)\in V^4 :\ \chi(v_1,v_2) = \chi(v_3,v_4) \right\} \right|.
\end{equation}

Color energy was introduced in \cite{PS18}, imitating the concept of additive energy.
Studying this quantity immediately led to new bounds for $f(n,k,\ell)$.
In the current work, we further push this technique in several different ways.
We first show how an Extremal Graph Theory bound for graphs with no cycles of length $k/2$ can be used to amplify uses of color energy.
As a warmup, we use this approach to derive Theorem \ref{th:OneFromQuad} in Section \ref{se:TwoCycles}.

We then introduce the concept of higher color energies.
The $r$-th color energy of a graph is a variant of $\E(G)$ that consists of $2r$-tuples instead of quadruples.
This concept is properly introduced in Section \ref{sec:HigherEnergy}, and is then used to prove Theorems \ref{th:ThreeCycles} and \ref{th:ArithLower}.

Section \ref{sec:NoEnergy} contains our proofs that do not rely on color energy.
In particular, it contains the proofs of Theorem \ref{th:threshold} and of Proposition \ref{prop:ArithLower}.

\parag{Acknowledgements.}
We would like to thank Yufei Zhao for a discussion that led to Proposition \ref{prop:ArithLower}.
We would also like to thank Robert Krueger and Rados Radoicic for several helpful discussions.

\section{Preliminaries}

In this section we describe several results that are used in our proofs.
We divide these results according to topics.

\parag{Extremal Graph Theory.}
The following is a classical result of K{\H o}vari, S\'os, and T\'uran (for example, see \cite[Section IV.2]{Bollobas98}).

\begin{lemma} \label{le:KST}
Let $G=(V_1\cup V_2,E)$ be a bipartite graph with $|V_1|=m$ and $|V_2|=n$.
If $G$ does not contain a copy of $K_{s,t}$, then
\[ |E| = O_{s,t}\left(mn^{1-\frac{1}{s}}+n\right). \]
\end{lemma}

A classical result about graphs with no cycles of a given length was originally stated by Erd\H os \cite{Erd63} without proof.
For an elegant proof, see Naor and Verstraete \cite{NV05}.

\begin{theorem} \label{th:Girth}
The following holds for every $k\ge 2$.
Any graph with $n$ vertices that does not contain $C_{2k}$ as a subgraph has $O\left(n^{1+1/k}\right)$ edges.
\end{theorem}

Given a graph $G=(V,E)$, a \emph{subdivision} of $G$ is the following bipartite graph.
The first vertex set of the subdivision contains a vertex for every element of $V$ and the second vertex set contains a vertex for every edge of $E$.
Every edge of the subdivision is between an edge $e\in E$ and a vertex $v\in V$ such that $v$ is an endpoint of $e$.
Let $H_t$ be the subdivision of $K_t$.
The following is a recent result of Janzer \cite{Janzer18}.

\begin{theorem} \label{th:Subdiv}
Let $t\ge 3$.
Any graph with $n$ vertices that contains no copy of $H_t$ has $O\left(n^{3/2-1/(4t-6)}\right)$ edges.
\end{theorem}

\parag{Probabilistic method.}
We now describe several lemmas that are based on standard probabilistic arguments (see for example \cite{AS04}).
We denote the expectation of a random variable $X$ as $\EE[X]$ (unfortunately, both $E$ and $\E$ are already in use).

\begin{lemma}\label{le:bipartite}
Consider a graph $G=(V,E)$ with $|V|=n\ge 100$.
Then $V$ can be partitioned into disjoint sets $V_1,V_2$ such that $|V_1|=\lceil n/2\rceil, |V_2|=\lfloor n/2 \rfloor$, and at least $|E|/3$ of the edges of $E$ do not have both of their endpoints in the same $V_j$.
\end{lemma}
\begin{proof}
We uniformly choose a partition of $V$ among the set of partitions satisfying $|V_1|=\lceil n/2\rceil$ and $|V_2|=\lfloor n/2 \rfloor$.
Consider an edge $e=(v,u)\in E$ and let $X_e$ be the indicator random variable stating whether the endpoints of $e$ are not both in the same $V_j$.
Since $n\ge 100$, it can be easily verified that $\prr[X_e] > 1/3$.
This in turn implies that $\EE[X_e] = \prr[X_e] > 1/3$.
Let $X$ denote the number of edges of $E$ that do not have both of their endpoints in the same $V_j$.
That is, $X= \sum_{e\in E} X_e$.
By linearity of expectation,
\[ \EE[X] = \sum_{e\in E} \EE[X_e] > |E|/3. \]

Since the expected size of $X$ is larger than $|E|/3$, there exists at least one partition for which $X>|E|/3$.
\end{proof}

We also derive a more unusual variant of the preceding lemma.

\begin{lemma}\label{le:Rpartite}
Consider an integer $r\ge 2$, a graph $G=(V,E)$ with $|V|=n$ for a sufficiently large $n$, and a set $T\subset E^r$.
Then $V$ can be partitioned into disjoint sets $V_1,\ldots,V_r$, each of size $\lceil n/r\rceil$ or $\lfloor n/r \rfloor$, such that $\Omega_r(|T|)$ of the tuples of $T$ consist only of edges having both of their endpoints in the same $V_j$.
\end{lemma}
\begin{proof}
We uniformly choose a partition $V_1,\ldots,V_r$ of $V$ among the set of partitions into $r$ parts of size $\lceil n/r\rceil$ or $\lfloor n/r \rfloor$.
Consider an $r$-tuple $t\in T$ and let $X_t$ be the indicator random variable stating whether every edge in $t$ has both of its endpoints in the same $V_j$ (different edges of $t$ may be in different parts).
Note that $t$ is defined using at most $2r$ vertices of $V$, and $X_t =1$ when all of these vertices are in the same part (this is a much stronger condition).
We claim that the probability of $2r$ vertices being in $V_1$ is at least $(4r)^{-2r}$.
Indeed, the total number of partitions of $V$ is $n!/((n/r)!)^{r}$ and the number of partitions where all $2r$ vertices are in $V_1$ is $(n-2r)!((n/r)!)^{-r+1}(n/r-2r)^{-1}$ (assuming for simplicity that $r$ divides $n$).
When $n$ is sufficiently large, the ratio between these two numbers is larger than $(4r)^{-2r}$.

By the above, for a sufficiently large $n$ we have $\prr[X_t] > (4r)^{-2r}$.
This in turn implies that $\EE[X_t] = \prr[X_t] > (4r)^{-2r}$.
Let $X$ denote the number of $r$-tuples of $T$ that contain only edges with both their endpoints in the same $V_j$.
That is, $X= \sum_{t\in T} X_t$.
By linearity of expectation,
\[ \EE[X] = \sum_{t\in T} \EE[X_t] > |T|(4r)^{-2r}. \]

Since the expected size of $X$ is larger than $|T|(4r)^{-2r}$, there exists at least one partition for which $X>|T|(4r)^{-2r}$.
\end{proof}

\section{Color energy with no cycles} \label{se:TwoCycles}

This section presents our simplest use of color energy --- the proof of Theorem \ref{th:OneFromQuad}.
This proof demonstrates how color energy can be combined with results concerning graphs with no cycles of a given length.
We first recall the statement of Theorem \ref{th:OneFromQuad}.
\vspace{2mm}

\noindent {\bf Theorem \ref{th:OneFromQuad}.}
\emph{For every $k \geq 8$ that is divisible by four, we have}
\[ f\left(n,k,\binom{k}{2} - k/2 + 1\right) = \Omega\left(n^{2 - 8/k}\right). \]
\begin{proof}
Consider a complete graph $K_n$ denoted as $G = (V, E)$, a set of colors $C$, and a function $\chi: E \to C$, such that every induced $K_{k}$ contains at least $\binom{k}{2} - k/2 + 1$ colors.
Consider also the color energy $\E(G)$, as defined in \eqref{eq:EnergyDeg}.
For a color $c\in C$, we set
\[ m_c = \left|\left\{(v_1,v_2)\in V^2 :\ v_1\neq v_2 \text{ and } \chi(v_1,v_2)=c \right\}\right|. \]
The number of quadruples $(v_1,v_2,v_3,v_4)\in V^4$ that satisfy $\chi(v_1,v_2) = \chi(v_3,v_4)=c$ is $m_c^2$.
This implies that
\[ \E(G) = \sum_{c\in C} m_c^2. \]

Since every ordered pair of distinct vertices in $V^2$ contributes to exactly one $m_c$, we get that $\sum_{c\in C}m_c = n(n-1)$. Combining the above with the Cauchy--Schwarz inequality leads to
\begin{equation} \label{eq:LowerEnergy}
\E(G) = \sum_{c\in C} m_c^2 \ge \frac{\left(\sum_{c\in C}m_c\right)^2}{|C|} = \frac{n^2(n-1)^2}{|C|}.
\end{equation}
The lower bound \eqref{eq:LowerEnergy} implies that, to obtain a lower bound on $|C|$ it suffices to derive an upper bound for $\E(G)$.

We define the \emph{energy graph} of $K_n$ to be the graph $G'=(V',E')$ defined as follows.
The set of vertices is $V' = V\times V$, including pairs where the same vertex appears twice.
An edge between $(v_1,v_3),(v_2,v_4)\in V'$ is in $E'$ if and only if $\chi(v_1,v_2)=\chi(v_3,v_4)$.
Note that $\E(G) = |E'|$.
Thus, we reduced the problem to deriving an upper bound for the number of edges in the energy graph of $K_n$.

We perform two stages of pruning $E'$.
First, we remove the $n(n-1)$ loops of the form $((a,b), (a,b))\in E'$ for some $(a,b) \in V'$.
This turns $G$ into a simple graph.
If this step removes more than half of the edges of $E'$, we have $\E(G) = |E'| = O(n^2)$.
In this case, \eqref{eq:LowerEnergy} implies $|C|=\Omega(n^2)$ and completes the proof.
We may thus assume that at most half of the edges of $E'$ were removed.

Recall that every edge $e\in E'$ corresponds to a pair of edges of $E$ that have the same color $c\in C$.
We associate $e$ with the color $c$.
For every color $c\in C$ that has fewer than $100k^2$ edges associated with it, we remove from $E'$ every edge that is associated with $c$.
If $\Omega\left(n^{2 - \frac{8}{k}}\right)$ colors of $C$ have $100k^2$ edges associated with them, then $|C|=\Omega\left(n^{2 - \frac{8}{k}}\right)$, which complete the proof.
We may thus assume that $O\left(n^{2 - \frac{8}{k}}\right)$ edges were removed from $E'$.
Once again, this does not change the asymptotic size of $E'$.

Since the above pruning steps did not change the asymptotic size of $|E'|$, we still have that $\E(G) = O(|E'|)$.
To bound $|E'|$, we wish to apply Theorem \ref{th:Girth} on $G'$.
For this purpose, we assume for contradiction that $G'$ contains a simple cycle $\gamma$ of length $k/2$.
In the statement of Theorem \ref{th:OneFromQuad}, $k$ is required to be divisible by four since Theorem \ref{th:Girth} only holds for cycles of even length.
We write $\gamma = (a_1, b_1), \cdots, (a_{k/2}, b_{k/2})$, where the vertices $a_1,\ldots, a_{k/2},b_1,\ldots,b_{k/2}\in V$ may not be distinct.
Let $S$ be the set of these vertices, so $|S|\le k$.

For some intuition, we first consider the case where $S$ consists of $k$ distinct vertices of $V$.
For every $1\le j \le k/2$, the edge between $(a_j,b_j)$ and $(a_{j+1},b_{j+1})$ implies that $\chi(a_j,a_{j+1})=\chi(b_{j},b_{j+1})$ (where $a_{k/2+1} = a_1$ and $b_{k/2+1}=b_1$).
This in turn implies that the number of distinct colors spanned by the vertices of $S$ is at most $\binom{k}{2}-k/2$.
We obtained a contradiction to the assumption that every $k$ vertices of $V$ span at least $\binom{k}{2}-k/2+1$ colors.

We next consider the general case, where $S$ might contain fewer than $k$ vertices of $V$.
We go one-by-one over the edges of $\gamma$.
In particular, in the $j$-th step we consider the edge between $(a_j,b_j)$ and $(a_{j+1},b_{j+1})$ (as before, $a_{k/2+1} = a_1$ and $b_{k/2+1}=b_1$).
At each step, we have $\chi(a_j,a_{j+1})=\chi(b_{j},b_{j+1})$ and this is either a new color repetition or a repetition we already counted in one of the previous steps.
If we are in the latter case, that means that both $a_j$ and $b_j$ already appeared in previous edges of the cycle.

Let $m$ mark the number of steps in which we did not find a new color repetition.
In other words, there are at least $k/2-m$ distinct color repetitions.
In each of the $m$ steps without a new repetition we also had two repeating vertices, so $|S|\le k-2m$.
Let $c=\chi(a_1,a_2)$.
We add to $S$ the endpoints of $m$ more edges with color $c$, and note that $|S|\le k$.
This is always possible, since by the graph pruning $c$ has at least $100k^2$ edges associated with it.
If necessary, we add to $S$ additional arbitrary vertices until it is of size $k$.
Since the vertices of $S$ span at most $\binom{k}{2}-k/2$ colors, we again obtain a contradiction.

The above contradiction implies that the pruned energy graph $G'$ does not contain a cycle of length $k/2$.
By Theorem \ref{th:Girth}, we obtain
\[ \E(G) = O(|E'|) = O\left(\left(n^{2}\right)^{1+4/k}\right) = O\left(n^{2+8/k}\right). \]
Combining this with \eqref{eq:LowerEnergy} implies the asserted bound $|C|=\Omega(n^{2-8/k})$.
\end{proof}

\section{Higher color energies}\label{sec:HigherEnergy}

In this section we study higher color energies.
As before, consider a copy of $K_n$ denoted as $G = (V, E)$, a set of colors $C$, and a function $\chi: E \to C$.
For an integer $r\ge 2$, we define the $r$-\textit{th color energy} of the graph as
\[ \E_r(G) = \left|\left\{ (a_1, a_2, \cdots, a_{2r} )\in V^{2r},\ \chi(a_1,a_2) = \chi(a_3,a_4) = \cdots = \chi(a_{2r-1},a_{2r}) \right\}\right|. \]

As before, for a color $c\in C$ we set $m_c = \left|\left\{(v_1,v_2)\in V^2 :\ v_1\neq v_2 \text{ and } \chi(v_1,v_2)=c \right\}\right|$.
We also recall that $\sum_{c\in C}m_c = n(n-1)$.
The number of $2r$-tuples that contribute to $\E_r(G)$ and correspond to the color $c$ is exactly $m_c^r$.
This implies that $\E_r(G) = \sum_{c \in C} m_c^r$.
By H\"older's inequality,
\begin{equation} \label{eq:LowerEnergyHigher}
\E_r(G) =  \sum_{c \in C} m_c^r \geq \frac{\left(\sum_{c \in C} m_c\right)^r}{\left(\sum_{c \in C}1\right)^{r-1}}  = \frac{n^{r}(n-1)^r}{|C|^{r-1}}.
\end{equation}

Note that the ``standard'' color energy $\E(G)$ is the second color energy $\E_2(G)$.
By \eqref{eq:LowerEnergyHigher}, to obtain a lower bound for the number of colors it suffices to derive an upper bound for $E_r(G)$, for some $r\ge 2$.

We now remove some edges from $E$ and update $\E_r(G)$ accordingly.
That is, after removing an edge $(u,v)\in E$, we also ignore all of the $2r$-tuples that involve $\chi(u,v)$.
By Lemma \ref{le:Rpartite}, there exist a partition of $V$ into $r$ disjoint subsets $V_1,\ldots, V_r$, each of size $\Theta(n)$, with the following property.
When removing from $E$ every edge that does not have both of its endpoints in the same $V_j$, the size of $\E_r(G)$ does not change asymptotically.
We indeed remove from $G$ every such edge.

An $r$\emph{-th energy graph} of $G$, denoted $G^*=(V^*,E^*)$, is defined as follows (each such graph corresponds to a different partition of $V$ into $V_1,\ldots, V_r$).
The set of vertices is $V^* = V_1\times V_2 \times \cdots \times V_r$.
An edge between $(v_1,\ldots,v_r),(v'_1,\ldots,v'_r)\in V_1\times \cdots \times V_r$ is in $E^*$ if and only if $\chi(v_1,v'_1)=\chi(v_2,v'_2)=\cdots=\chi(v_r,v'_r)$.
Note that $\E_r(G) = \Theta(|E^*|)$.
Thus, to obtain a lower bound for the number of colors, we can derive an upper bound on the number of edges in an $r$-th energy graph of $K_n$.
An $r$-th energy graph is not a simple graph since each of its $n^r$ vertices forms a loop.
These loops correspond in $\E_r(G)$ to tuples of the form $(a_1,\ldots,a_r,a_1,\ldots,a_r)$.
We will assume that every $r$-th energy graph is simple by removing these $n^r$ loops.
The number of removed edges is negligible and does not affect any of our proofs.

Finally, we remove ``unpopular'' colors from $G^*$, as follows.
Every edge $e\in E^*$ corresponds to several edges of $E$ that have the same color.
We say that $e$ also has this color.
For every color $c\in C$ that appears in $E$ at most $\log n$ times, we remove from $E^*$ every edge that is associated with $c$.
Note that every such color is associated with at most $\log^r n$ edges of $E^*$.
Since $|C|=O(n^2)$, this step removes $O(n^2\log^r n)$ edges from the $r$-th energy graph.
This number is too small to have an effect on any of our proofs.

We are now ready to prove our lower bound for $g(n,k,\ell)$.
\vspace{2mm}

\noindent {\bf Theorem \ref{th:ArithLower}.}
\emph{For all $k > r \geq 2$,}
\begin{align*}
g\left(n,2rk,\binom{2rk}{2} - \binom{2k}{2}\cdot\left[ \binom{r}{2} + (r-1) \right] + 1\right) &= \Omega\left(n^{\frac{r}{r-1}\cdot \frac{k-1}{k}}\right).
\end{align*}
\begin{proof}
Let $A$ be a set of $n$ real numbers such that every subset $A'\subset A$ of size $2rk$ satisfies $|A'-A'|\ge \binom{2rk}{2} - \binom{2k}{2}\cdot\left[ \binom{r}{2} + (r-1) \right] + 1$.
Let $G= (V,E)$ be a copy of $K_n$ with a vertex corresponding to each element of $A$.
We associate a color with each element of $A-A$ and color an edge $(a,a')$ with the color associated with $|a-a'|$ (recall that we define $A-A$ as containing only positive differences).
Let $C$ be the set of colors and note that $|C|=|A-A|$.
By \eqref{eq:LowerEnergyHigher}, to obtain a lower bound for $|C|$ it suffices to derive an upper bound for $\E_r(G)$.
Let $G^*=(V^*,E^*)$ be an $r$-th energy graph of $G$.
By the discussion before this proof, it suffices to derive an upper bound on $|E^*|$.

Consider an edge $((v_1,\ldots,v_r),(v'_1,\ldots,v'_r))\in E^*$.
Thinking of the vertices $v_j\in V$ as their corresponding elements in $A$, we have $|v_1-v'_1| = |v_2-v'_2| = \cdots = |v_r-v'_r|$.
We associate this edge with a sequence of $r-1$ symbols from $\{+,-\}$, as follows (that is, we associate the edge with an element of $\{+,-\}^{r-1}$).
For every $2\le j \le r$, the $(j-1)$-th element of the sequence is `+' if $v_1-v'_1 = v_j-v'_j$, and `-' if $v_1-v'_1 = v'_j-v_j$.
That is, the associated symbol encodes how to remove the absolute values from $|v_1-v'_1| = |v_2-v'_2| = \cdots = |v_r-v'_r|$.

We partition $G^*$ into $2^{r-1}$ graphs, as follows.
Each graph contains the same set of vertices $V^*$, and each graph corresponds to one of the $2^{r-1}$ sequences of $\{+,-\}^{r-1}$.
A graph that corresponds to a specific sequence $s\in \{+,-\}^{r-1}$ contains the edges of $E^*$ that are associated with $s$.
Note that every edge of $G^*$ corresponds to exactly one of the $2^{r-1}$ graphs.
Thus, to obtain an upper bound for $|E^*|$ it suffices to bound the number of edges in each of these graphs.

Let $H = (V^*,E_H)$ be one of the $2^{r-1}$ graphs constructed in the preceding paragraph.
Assume for contradiction that $H$ contains a cycle $\gamma$ of length $2k$, and denote the vertices of $\gamma$ as
\[ (v_{1,1}, \ldots, v_{1,r}), (v_{2,1}, \ldots, v_{2,r}), \ldots, (v_{2k,1}, \ldots, v_{2k,r} ). \]

Recall that every edge of $E_H$ corresponds to the same symbol $s\in \{+,-\}^{r-1}$.
By the definition of an edge in an energy graph, for every $1\le j \le 2k$ we have
\begin{equation} \label{eq:arithCycleEdge}
v_{j,1} - v_{j+1,1} = \pm( v_{j,2} - v_{j+1,2}) = \cdots = \pm(v_{j,r} - v_{j+1,r}),
\end{equation}
where $v_{2k+1,\ell}=v_{1,\ell}$ and each $\pm$ is replaced with either `+' or `-' according to $s$.

By the definition of $H$, the $\pm$ symbols are replaced in the same way for every edge of $\gamma$.
For any $1\le j_1 < j_2 \le 2k$, summing \eqref{eq:arithCycleEdge} for $j_1 \le j < j_2$ yields
\begin{equation}\label{eq:arithCycleNonEdge}
v_{j_1,1} - v_{j_2,1} = \pm(v_{j_1,2} - v_{j_2,2}) = \cdots = \pm (v_{j_1,r} - v_{j_2,r}).
\end{equation}

By the above, the vertices of the cycle $\gamma$ form a clique $K_{2k}$ in $G^*$.
We next claim that the $2kr$ vertices $v_{j,\ell}\in V$ used to define the vertices of $\gamma$ are all distinct.
By the definition of $V_1,\ldots,V_r$, if $j\neq j'$ then $v_{j,\ell}$ and $v_{j',\ell'}$ must correspond to different elements of $V$.
Assume for contradiction that $v_{j,\ell}=v_{j,\ell'}$ for some $\ell\neq \ell'$.
By \eqref{eq:arithCycleNonEdge} with $j_1=\ell$ and $j_2=\ell'$, we obtain that $(v_{\ell,1}, \ldots, v_{\ell,r}) = (v_{\ell',1}, \ldots, v_{\ell',r})$.
This contradicts $\gamma$ being a simple cycle, which implies that the $2kr$ vertices $v_{j,\ell}\in V$ are indeed distinct.

Consider the set $S$ consisting of the $2kr$ vertices $v_{j,\ell}\in V$ used to define the vertices of $\gamma$.
By the preceding paragraph $|S|=2kr$.
By \eqref{eq:arithCycleEdge}, for each of the $\binom{2k}{2}$ choices for $j_1$ and $j_2$ we have $r-1$ distinct color repetitions.
Consider one such repetition $v_{j_1,\ell} - v_{j_2,\ell} = v_{j_1,\ell'} - v_{j_2,\ell'}$ with $\ell\neq \ell'$ and note that it leads to a second repetition $v_{j_1,\ell} - v_{j_1,\ell'} = v_{j_2,\ell} - v_{j_2,\ell'}$ (if instead we start with $v_{j_1,\ell} - v_{j_2,\ell} =  v_{j_2,\ell'} - v_{j_1,\ell'}$ then we have the second repetition $v_{j_1,\ell} - v_{j_2,\ell'} =  v_{j_2,\ell} - v_{j_1,\ell'}$).
Thus, for each of the $\binom{2k}{2}$ choices for $j_1$ and $j_2$ we actually have $r-1 + \binom{r}{2}$ distinct color repetitions.
This contradicts the local property assumption, so $H$ does not contain a cycle of length $2k$.

Since $H$ does not contain a cycle of length $2k$, Theorem \ref{th:Girth} implies
\[ |E_H| = O\left(|V^*|^{1+1/k}\right) = O\left(n^{r+r/k}\right).  \]
Recall that $E^*$ is partitioned into $2^{r-1}$ subsets, each satisfying the above upper bound for $|E_H|$.
We thus have
\[ \E_r(G) = \Theta\left(|E^*|\right) = O\left(n^{r+r/k}\right). \]

Combining this upper bound for $\E_r(G)$ with \eqref{eq:LowerEnergyHigher} gives
\[ |A-A| = |C|=\Omega\left(n^{\frac{r}{r-1}\cdot \frac{k-1}{k}}\right). \]
\end{proof}

We also rely on higher color energy to prove the following result.
\vspace{2mm}

\noindent {\bf Theorem \ref{th:ThreeCycles}.}
$\displaystyle  f\left(n,24,\binom{24}{2} - 15\right) = \Omega\left(n^{9/8}\right)$.
\begin{proof}
Let $G=(V,E)$ be a copy of $K_n$, let $C$ be a set of colors, and let $\chi:E \to C$, such that every copy of $K_{24}$ in $G$ has at least $\binom{24}{2} - 15$ distinct colors.
Let $G^* = (V^*,E^*)$ be a third energy graph of $G$.

Assume for contradiction that there exists a vertex $v\in V$ adjacent to at least 17 edges of color $c\in C$.
Let $S$ be a set consisting of $v$, of 17 vertices that form with $v$ an edge of color $c$, and of six arbitrary additional vertices of $V$.
Then $S$ is a set of 24 vertices of $V$ that span at most $\binom{24}{2} - 16$ colors of $C$.
This contradicts the local property, so no vertex of $V$ can be adjacent to 17 edges of the same color.

\parag{Pruning.} We perform two steps of pruning $E^*$.
First, for every $1\le j \le 3$ we partition $V_j$ into two disjoint sets $V'_j$ and $V''_j$ and discard from $E^*$ every edge that has both of its endpoints containing a vertex from $V'_j$ or both of its endpoints containing a vertex from $V''_j$.
By imitating the proof of Lemma \ref{le:bipartite}, we get that this can be done without asymptotically changing the size of $E^*$.

\begin{figure}[h]
\centerline{\placefig{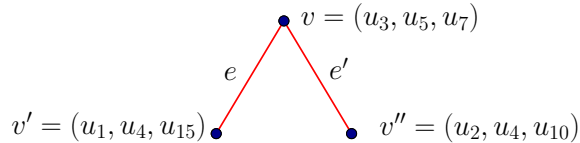}{0.46\textwidth}}
\vspace{-1mm}

\caption{\small \sf After deciding to keep $e$, we need to remove $e'$ because of the common coordinate $u_4$. }
\label{fi:CommonNeighbor}
\vspace{-2mm}
\end{figure}

In our second pruning step, we throw from $E^*$ edges until $G^*$ satisfies the following property: For every $v\in V^*$, no two neighbors of $v$ have the same value in one of their three coordinates.
In particular, we repeatedly choose an edge $e\in E^*$ that will remain in the graph and then remove every edge $e'\in E^*$ that violates the above condition together with $e'$.
Figure \ref{fi:CommonNeighbor} depicts a situation where after deciding to keep $e=(v,v')$, we need to remove $e'=(v,v'')$, since $v$ and $v'$ have the common coordinate $u_4\in V$.
We now show that this process does not asymptotically decrease the size of $E^*$.

Consider an edge $e=(v,v')$ that we decided to keep in $E^*$.
Write $v=(u_1,u_2,u_3)$ and $v'=(u'_1,u'_2,u'_3)$, where $u_1,u_2,u_3,u'_1,u'_2,u'_3\in V$.
Let $c= \chi(u_1,u'_1)=\chi(u_2,u'_2)=\chi(u_3,u'_3)$.
Assume for contradiction that $v$ has 17 neighbors of the form $(u'_1,u'_2,*)$, where the $*$ could be any vertex of $V_3$.
Since each of these 17 neighbors has a different vertex replacing the $*$, and each such vertex $w\in V_3$ must satisfy $\chi(u_3,w)=c$, we get that $u_3$ is adjacent to 17 edges of color $c$ in $G$.
This contradicts the above, so $v$ has at most 16 neighbors of the form $(u'_1,u'_2,*)$

We now assume that $v$ has more than $16^2$ neighbors of the form $(u'_1,*,*)$, where the $*$ symbols represent any $w_2\in V_2$ and $w_3\in V_3$, respectively.
By the preceding paragraph, after fixing $w_2$ there are at most 16 options for $w_3$.
Thus, there are at least 17 distinct values for $w_2$.
As in the preceding paragraph, this implies that $u_2$ is adjacent to at least 17 edges of color $c$.
This contradiction implies that $v$ has at most $16^2$ neighbors of the form $(u'_1,*,*)$.
Similarly, $v$ has at most $16^2$ neighbors of the form $(*,u'_2,*)$ and at most $16^2$ neighbors of the form $(*,*,u'_3)$.
Therefore, after keeping an edge in $E^*$ we remove at most $3\cdot 16^2$ edges adjacent to $v$.
Symmetrically, we remove at most $3\cdot 16^2$ edges adjacent to $v'$.
We conclude that the second pruning step does not change the asymptotic size of $E^*$.

\parag{Cycles in the pruned energy graph.}
Assume for contradiction that $G^*$ contains a cycle $\gamma$ of length eight.
We denote the vertices of $\gamma$ as $(a_1, b_1, c_1), \cdots, (a_8, b_8, c_8)$.
We create a set of vertices $S\subset V$ in eight steps, where during the $j$'th step we add $a_j,b_j,c_j$ to $S$ (it is possible that some of these vertices were already placed in $S$ in a previous step).
We now show that at each step we can add at most three vertices to $S$ and obtain at least two new color repetitions in the subgraph induced by $S$.
When beginning the $j$-th step, if at least two of the vertices $a_j,b_j,c_j$ are not already in $S$ then $\chi(a_j,a_{j+1})=\chi(b_j,b_{j+1})=\chi(c_j,c_{j+1})$ yields two new color repetitions.
As usual, we set $a_9=a_1, b_9=b_1,$ and $c_9=c_1$.

At the beginning of the first step all three vertices $a_1,b_1,c_1$ are new, since $S$ is empty.
Recalling the partitioning of each of $V_1,V_2,V_3$ into two disjoint sets, we note that $a_2,b_2,c_2$ are all new, since they cannot be identical to vertices from the first step.
Recalling also that the neighbors of a vertex of $V^*$ cannot have any identical coordinates, we get that $a_3,b_3,c_3,a_4,b_4,c_4$ are also all new.
That is, in the first four steps we place 12 vertices in $S$ and have eight distinct color repetitions.

If at step $j$ exactly one of $a_j,b_j,c_j$ is not already in $S$, then the edge involving this new vertex yields one new color repetition.
Since we only add to $S$ one new vertex from $a_j,b_j,c_j$, we are allowed to add two additional vertices.
We take an arbitrary edge of color $\chi(a_1,a_2)$ that is not already in the subgraph induced by $S$, and add both of the endpoints of this edge to $S$.
Such an edge is always available by the assumption that every color appears at least $\log n$ times (see the above definition of an $r$-th energy graph).
Note that in this case we indeed added at most three new vertices to $S$ and at least two new color repetitions to the subgraph induced by $S$.

In the fifth step, at least one of the vertices $a_5,b_5,c_5$ is not in $S$ yet.
Indeed, since each of the sets $V_1,V_2,V_3$ is bipartite, these vertices cannot be equivalent to vertices from steps with even indices.
Since neighbors with common coordinates are not allowed, $a_5\neq a_3, b_5\neq b_3,$ and $c_5\neq c_3$.
Finally, it is impossible to have $(a_1,b_1,c_1) = (a_5,b_5,c_5)$ since then we would have a vertex of $V^*$ repeating twice in $\gamma$, implying that $\gamma$ is not a simple cycle.
A similar argument shows that there is at least one new vertex of $V$ also in the sixth, seventh, and eighth steps.
This concludes the construction of $S$.

If after the above process $S$ still has fewer than 24 vertices, we keep adding arbitrary vertices until $|S|=24$.
By the above, the subgraph induced by $S$ has at least 16 color repetitions.
In other words, this subgraph contains at most $\binom{24}{2}-16$ distinct colors.
This contradicts the local property, so $G^*$ cannot contain a cycle of length eight.
Theorem \ref{th:Girth} implies
\[ \E_3(G) = \Theta\left(|E^*|\right)=O\left(|V^*|^{5/4}\right) = O\left(n^{15/4}\right). \]
Combining this with \eqref{eq:LowerEnergyHigher} (when $r=3$) yields $|C|=\Omega(n^{9/8})$, as asserted.

\end{proof}

\section{Proofs with no color energy} \label{sec:NoEnergy}

This section contains proofs where we do not rely on color energy.
We repeat each result before proving it.
\vspace{2mm}

\noindent {\bf Theorem \ref{th:threshold}.}
\emph{(a) For any integers $2 \leq m \leq k/2$,}
\[ f\left(n,k,\binom{k}{2} - m(k - m) + 2\right) = \Omega\left(n^{1/m}\right). \]

\emph{(b) For any integer $t\ge3$ and $k= \binom{t+1}{2}$,  }
\[ f\left(n,k,\binom{k}{2} - t(t-1) +1 \right) = \Omega\left(n^{1/2+1/(4t-6)}\right). \]

\begin{proof}
(a) Consider a copy of $K_n = (V,E)$, a set of colors $C$, and a function $\chi:E\to C$, such that every copy of $K_k$ in the graph contains at least $\binom{k}{2} - m(k - m) + 2$ colors.
Assume for contradiction that there exists a color $c\in C$ such that $\Omega(n^{2-1/m})$ edges $e\in E$ satisfy $\chi(e)=c$ (with a sufficiently large constant in the $O(\cdot)$-notation).
Let $E' \subset E$ be the set of edges with color $c$.

By Lemma \ref{le:bipartite}, we can partition the vertices of $V$ into two disjoint sets $V_1,V_2$ each of size $\Theta(n)$ such that $\Theta(|E'|)$ of the edges of $E'$ have one endpoint in $V_1$ and one in $V_2$.
Let $E^*\subset E'$ denote the set of edges with one endpoint in each set.
Then $G=(V_1,V_2, E^*)$ is a bipartite graph with $\Theta(n)$ vertices in each part and $\Omega\left(n^{2-1/m}\right)$ edges.
Since we started with a sufficiently large constant in the $\Omega(\cdot)$-notation, by Lemma \ref{le:KST} we get that $G$ contains a copy of $K_{m,k-m}$.

From the preceding paragraph, we have that the original colored $K_n$ contains a copy of $K_{m,k-m}$ with all of its edges having color $c$.
This is a set of $k$ vertices with at most $\binom{k}{2} - m(k-m)+1$ distinct colors.
Since this contradicts the local property of the coloring, we conclude that every color of $C$ appears $O\left(n^{2-1/m}\right)$ times in $k_n$.
This immediately implies that $|C|=\Omega\left(n^{1/m}\right)$.

(b) The proof is obtained by repeating the proof of part (a), while replacing the use of Lemma \ref{le:KST} with Theorem \ref{th:Subdiv}.
That is, we assume for contradiction that there exists a color $c\in C$ such that $\Omega\left(n^{3/2-1/(4t-6)}\right)$ edges $e\in E$ satisfy $\chi(e)=c$.
We then obtain a contradiction by showing that the colored $K_n$ contains a copy of $H_t$ with all of its edges having the color $c$.
A copy of $H_t$ consists of $\binom{t+1}{2}$ vertices and $t(t-1)$ edges.
\end{proof}

For the following result, see for example \cite[Theorem 9.1 of Chapter 2]{Rusza09} combined with \cite{Bloom16}.

\begin{theorem}\label{th:No3Term}
Let $\eps>0$ and let $A$ be a set of $n$ elements that contains no 3-term arithmetic progression.
Then $|A-A|=\Omega\left(n \cdot \log^{1/4-\eps} n\right)$.
\end{theorem}

We are now ready to prove our upper bound for $g(n,k,\ell)$.
\vspace{2mm}

\noindent {\bf Proposition \ref{prop:ArithLower}.}
\emph{For every $\eps > 0$, any sufficiently large $c$ satisfies the following.
For every sufficiently large integer $k$,}
\[ g\left(n,k,c \cdot k \cdot \log^{1/4 - \eps} k\right) = n2^{O(\sqrt{\log n})}. \]
\begin{proof}
Behrend \cite{Behrend46} proved that there exists a set $A$ of $n$ positive integers such that $|A-A|= n\cdot 2^{O(\sqrt{\log n})}$ and no three elements of $A$ form an arithmetic progression.
Let $B \subseteq A$ satisfy $|B| = k$.
Since $B$ does not contain a 3-term arithmetic progression and $k$ is sufficiently large, Theorem \ref{th:No3Term} gives
\[ |B - B| =\Omega(k \log^{1/4-\eps} k) \le c \cdot k \cdot \log^{1/4-\eps} k, \]
where the last transition holds for sufficiently large $c$.
This construction immediately implies the asserted bound.
\end{proof}

\end{document}